\documentclass{svproc}
\usepackage{dsfont}
\usepackage{url}
\usepackage[utf8]{inputenc}
\usepackage[T1]{fontenc}
\usepackage{lmodern}
\usepackage{mathtools}
\usepackage{amssymb}
\usepackage{lipsum}
\usepackage{mathrsfs}
 \usepackage{mathabx}
\usepackage{color}
\usepackage{ulem}

\makeindex             

\begin{document}

\title{An inverse conductivity problem in multifrequency electric impedance tomography \vskip .3cm {\small\tt Dedicated to Masahiro Yamamoto for his sixtieth birthday}}
\titlerunning{Inverse conductivity problem}

\author{Jin Cheng\inst{1} \and Mourad Choulli \inst{2} \and Shuai Lu \inst{1}}
\authorrunning{Jin Cheng, Mourad Choulli and Shuai Lu} 
%
\tocauthor{Jin Cheng, Mourad Choulli, and Shuai Lu}
\institute{Shanghai Key Laboratory for Contemporary Applied Mathematics, \\ Key Laboratory of Mathematics for Nonlinear Sciences and \\ School of Mathematical Science, Fudan University, 200433 Shanghai, \email{jcheng@fudan.edu.cn, slu@fudan.edu.cn}
\and
Universit\'e de Lorraine, 34 cours L\'eopold, 54052 Nancy cedex, France \email{mourad.choulli@univ-lorraine.fr}}

\maketitle

\abstract{We deal with the problem of determining the shape of an inclusion embedded in a homogenous background medium. The multifrequency electrical impedance tomography is used to image the inclusion. For different frequencies, a current is injected at the boundary and  the resulting potential is measured. It turns out that the potential solves an elliptic equation in divergence form with discontinuous leading coefficient. For this inverse problem we aim to establish a logarithmic type stability estimate. The key point in our analysis consists in reducing the original problem to that of determining an unknown part of the inner boundary from a single boundary measurment. The stability estimate is then used to prove uniqueness results. We also provide an expansion of the solution of the BVP under consideration in  the eigenfunction basis of the Neumann-Poincar\'e operator associated to the Neumann-Green function.}


\section{Introduction}

We firstly proceed with the mathematical formulation of the problem under consideration. In order to specify the BVP satisfied by the potential, we consider $\Omega$ and $D$ two Lipschitz domains of $\mathbb{R}^n$, $n\ge 2$, so that $D\Subset \Omega$. Fix $k_0>0$ and define, for $k\in (0,\infty)\setminus\{k_0\}$, the function $\mathbf{a}_D$ on $\Omega$ by
\[
\mathbf{a}_D(k)=k_0+(k-k_0)\chi_D.
\]
Here $\chi_D$ is the characteristic function of the inclusion $D$, $k_0$ is the conductivity of the background medium and $k$ is the conductivity of the inclusion.

It is well known that in the present context the potential solves the BVP
\begin{equation}\label{1.1}
\left\{
\begin{array}{ll}
\mbox{div}\left(\mathbf{a}_D(k)\nabla u\right)=0\quad &\mbox{in}\; \Omega ,
\\
k_0\partial_\nu u=f &\mbox{on}\; \partial \Omega ,
\end{array}
\right.
\end{equation}
where $\partial_\nu$ is the derivative along the unit normal vector field $\nu$ on $\partial \Omega$ pointing outward $\Omega$.

Let
\[
V=\left\{ u\in H^1(\Omega );\; \int_{\partial \Omega}u(x)d\sigma(x)=0\right\}.
\]
Note that $V$ is a Hilbert space when it is endowed with the scalar product
\[
\mathfrak{a}(u,v)=\int_\Omega \nabla u\cdot \nabla v,\quad u,v\in V.
\]
We leave to the  reader to check that the norm induced by this scalar product is equivalent to the $H^1(\Omega )$-norm on $V$.

Pick $f\in L^2(\partial D)$ and $k\in (0,\infty)\setminus\{k_0\}$. Then it is not hard to see that, according to Lax-Milgram's lemma,  the BVP \eqref{1.1} possesses a unique variational solution $u_D(k)\in V$. That is $u_D(k)$ is the unique element of $V$ satisfying
\begin{equation}\label{1.2}
\int_\Omega \mathbf{a}_D(k)\nabla u_D(k)\cdot \nabla vdx=\int_{\partial \Omega}fvd\sigma (x),\quad v\in V.
\end{equation}

Recall that the trace operator $\mathfrak{t}:w\in C^\infty (\overline{\Omega})\mapsto u_{|\partial \Omega}\in C^\infty (\partial \Omega )$ is extended to a bounded operator, still denoted by $\mathfrak{t}$, from $H^1(\Omega )$ into $L^2(\partial \Omega )$.

Since we will consider conductivities varying with the frequency $\omega$, we introduce the map $k:\omega \in (0,\infty )\mapsto k(w)\in (0,\infty)\setminus\{k_0\}$. We assume in the sequel the condition
\[
\lim_{\omega \rightarrow \infty}k(\omega)=\infty .
\]

Let $\mathfrak{J}$ be a given subset of $(0,\infty)$. In the present work, we are mainly interested in determining the unknown subdomain $D$ from the boundary measurements
\[
\mathfrak{t}u_D(k(\omega )),\quad \omega \in \mathfrak{J}.
\]

Prior to state our main result, we introduce some notations and definitions. Let
\begin{align*}
&\mathbb{R}_+^n=\{x=(x',x_n)\in \mathbb{R}^n;\; x_n>0\},
\\
&Q=\{x=(x',x_n)\in \mathbb{R}^n;\; |x'|<1\; \mbox{and}\; |x_n|<1\},
\\
&Q_+=Q\cap \mathbb{R}^n_+,
\\
&Q_0=\{x=(x',x_n)\in \mathbb{R}^n;\; |x'|<1\; \mbox{and}\; x_n=0\}.
\end{align*}
Fix $0<\alpha <1$. We say that the bounded domain $\mathcal{U}$ of $\mathbb{R}^n$ is of class $C^{2,\alpha}$ with parameters $\varrho>0$ and $\aleph >0$ if for any $x\in \partial\mathcal{U}$ there exists a bijective map $\phi: Q\rightarrow B:=B_{\mathbb{R}^{n-1}}(x,\varrho)$, satisfying $\phi\in C^{2,\alpha} (\overline{Q})$, $\phi^{-1}\in C^{2,\alpha}(\overline{B})$ and
\[
\|\phi\|_{C^{2,\alpha} (\overline{Q})}+\|\phi^{-1}\|_{C^{2,\alpha} (\overline{B})}\le \aleph ,
\]
so that
\[
\phi (Q_+)=\mathcal{U}\cap B\quad \mbox{and}\quad \phi(Q_0)=B\cap \partial \mathcal{U}.
\]
If we substitute in this definition $C^{2,\alpha}$ by $C^{0,1}$ then we obtain the definition of a Lipschitz domain with parameters $\varrho$ and $\aleph$.

Let $\mathscr{D}_0(\varrho,\aleph)$ denote the set of  subdomains $D$ so that $D\Subset \Omega$, and  $\Omega \setminus\overline{D}$ is of class $C^{2,\alpha}$ with parameters $\varrho>0$ and $\aleph >0$.


Fix $D_0\Subset \Omega_0 \Subset \Omega$ and $\delta >0$. Consider then $\mathscr{D}_1(\varrho,\aleph ,\delta)$ the set of subdomains $D\Subset \Omega_0\Subset \Omega$ satisfying $D_0\Subset D$, $\Omega \setminus\overline{\Omega_0}$ and $\Omega_0\setminus\overline{D}$ are domains of class $C^{2,\alpha}$ with parameters $\varrho>0$ and $\aleph >0$,
\[
\mbox{dist}(D,\partial \Omega_0)\ge \delta ,
\]
and the following assumption holds
\begin{equation}\label{As}
B(x_0,d(x_0,\overline{D}))\subset \Omega \setminus\overline{D},\quad x_0\in \Omega_0\setminus \overline{D}.
\end{equation}

Henceforward, we make the assumption that $\Omega_0$, $D_0$ and $\delta$ are chosen in such a way that $\mathscr{D}_1(\varrho,\aleph ,\delta)$ is nonempty.

Define also $\mathscr{C}_0(\varrho, \aleph,\varrho_0,\aleph_0,\delta )$ as the set of couples $(D_1,D_2)$ so that $D_j\in \mathscr{D}_1(\varrho,\aleph,\delta) $, $j=1,2$, and $\Omega \setminus \overline{D_1\cup D_2}$ is a domain of class $C^{0,1}$ with parameters $\varrho_0>0$ and $\aleph_0 >0$.

Define the geometric distance $d_g^U$ on a bounded domain $U$ of $\mathbb{R}^n$ by
\[
d_g^U(x,y)=\inf\left \{ \ell (\psi ) ;\; \psi :[0,1]\rightarrow U \; \mbox{is Lipschitz path joining}\; x \; \mbox{to}\; y\right\},
\]
where
\[
\ell (\psi )= \int_0^1\left|\dot{\psi}(t)\right|dt
\]
is the length of $\psi$.

Note that, according to Rademacher's theorem, any Lipschitz continuous function $\psi :[0,1]\rightarrow D$ is almost everywhere differentiable with $\left|\dot{\psi}(t)\right|\le L$ a.e. $t\in [0,1]$, where $L$ is the Lipschitz constant of $\psi$. Therefore, $\ell (\psi )$ is well defined.

From \cite[Lemma 3.3]{CY}, we know that $d_g^U\in L^\infty (U \times U )$ whenever $U$ is of class $C^{0,1}$. Fix then $\mathfrak{b}>0$ and define $\mathscr{C}_1(\varrho, \aleph,\varrho_0,\aleph_0,\delta ,\mathfrak{b})$ as the subset of the couples $(D_1,D_2)\in \mathscr{C}_0(\varrho, \aleph,\varrho_0,\aleph_0,\delta)$ so that
\[
d_g^{\Omega \setminus\overline{D_1\cup D_2}}\le \mathfrak{b}.
\]

Recall that the Hausdorff distance for compact subsets of $\mathbb{R}^n$ is given by
\[
\mathbf{d}_H(\overline{D}_1,\overline{D_2})=\max\left( \max_{x\in \overline{D}_1}d(x,\overline{D}_2),\max_{x\in \overline{D}_2}d(x,\overline{D}_1)\right)
\]
and, following \cite{ABRV}, we define the modified distance $\mathbf{d}_m$ by
\[
\mathbf{d}_m(\overline{D}_1,\overline{D_2})=\max \left( \max_{x\in \partial D_1}d(x, \overline{D}_2), \max_{x\in \partial D_2}d(x,\overline{D}_1)\right).
\]

As it is pointed out in \cite{ABRV}, $\mathbf{d}_m$ is not a distance. To see it, consider $D_1=B(0,1)\setminus B(0,1/2)$ and $D_2=B(0,1)$. In that case simple computations show that
\[
0=\mathbf{d}_m(\overline{D}_1,\overline{D_2})<\mathbf{d}_H(\overline{D}_1,\overline{D_2})=\frac{1}{2}.
\]

In this example $D_1\subset D_2$, but $\overline{D}_1\not\subset D_2$. However, we can enlarge slightly $D_2$ in order to satisfy $\overline{D}_1\subset D_2$. Indeed, if $D_2=B(0, 5/4)$ then $\overline{D}_1\subset D_2$ and
\[
\frac{1}{4}=\mathbf{d}_m(\overline{D}_1,\overline{D_2})<\mathbf{d}_H(\overline{D}_1,\overline{D_2})=\frac{1}{2}.
\]
In dimension two, take $D_1=\{ (r,\theta );\; \eta <\theta <2\pi -\eta ,1/2<r<1\}$, $0<\eta <\pi$. Smoothing the angles of $D_1$ we get a $C^\infty$ simply connected domain so that, if again $D_2=B(0, 5/4)$,
\[
\frac{1}{4}\le \mathbf{d}_m(\overline{D}_1,\overline{D_2})<\mathbf{d}_H(\overline{D}_1,\overline{D_2})=\frac{1}{2}.
\]

In all these examples a small translation in the direction of one of the coordinates axes for instance enables us to construct examples with $D_1\setminus \overline{D}_2\ne \emptyset$, $D_1\setminus \overline{D}_2\ne \emptyset$ and $\mathbf{d}_m(\overline{D}_1,\overline{D_2})<\mathbf{d}_H(\overline{D}_1,\overline{D_2})$.

What these examples show is that it is difficult to give sufficient geometric conditions ensuring that the following equality holds
\begin{equation}\label{ED}
\mathbf{d}_m(\overline{D}_1,\overline{D_2})=\mathbf{d}_H(\overline{D}_1,\overline{D_2}).
\end{equation}
It is obvious to check that \eqref{ED} is satisfied whenever $D_1$ and $D_2$ are balls or ellipses, but not only.

The subset  of couples $(D_1,D_2)\in\mathscr{C}_2(\varrho, \aleph,\varrho_0,\aleph_0,\delta, \mathfrak{b})$ satisfying \eqref{ED}  will denoted $\mathscr{C}_1(\varrho, \aleph,\varrho_0,\aleph_0,\delta ,\mathfrak{b})$.

We fix in all of this text $f\in C^{1,\alpha}(\partial \Omega)$ non negative and  non identically equal to zero.

We aim in this paper to establish the following result.

\begin{theorem}\label{theorem1.1}
Let $\mathfrak{d}=(\varrho, \aleph,\varrho_0,\aleph_0,\delta ,\mathfrak{b})$. There exist two constant $C=C(\mathfrak{d})$ and $0<\Lambda^\ast =\Lambda ^\ast(\mathfrak{d}) <e^{-e}$ so that, for any $(D_1,D_2)\in \mathscr{C}_2(\mathfrak{d})$ with $\partial D_1\cap \partial D_2\ne \emptyset$ and any sequence $(k_j)\in (0,\infty )$ satisfying $k_j\rightarrow \infty$, we have
\[
\mathbf{d}_H(\overline{D}_1,\overline{D_2})\le C\left(\ln \ln |\ln \Lambda |\right)^{-1},
\]
provided that
\[
0<\Lambda := \sup_j\| \mathfrak{t}u_{D_2}(k_j)-\mathfrak{t}u_{D_1}(k_j)\|_{L^2 (\partial \Omega) }<\Lambda^\ast .
\]
\end{theorem}

As an immediate consequence of this theorem we have the following corollary.

\begin{corollary}\label{corollary1.1}
Let $\mathfrak{d}=(\varrho, \aleph,\varrho_0,\aleph_0,\delta,\mathfrak{b})$. There exist two constant $C=C(\mathfrak{d})$ and $0<\Lambda^\ast =\Lambda ^\ast(\mathfrak{d}) <e^{-e}$ so that, for any $(D_1,D_2)\in \mathscr{C}_2(\mathfrak{d})$ with $\partial D_1\cap \partial D_2\ne \emptyset$ and any sequence of frequencies $(\omega_j)\in (0,\infty )$ satisfying $\omega_j\rightarrow \infty$, we have
\[
\mathbf{d}_H(\overline{D}_1,\overline{D_2})\le C\left(\ln \ln |\ln \Lambda |\right)^{-1},
\]
provided that
\[
0<\Lambda := \sup_j\| \mathfrak{t}u_{D_2}(k(\omega_j))-\mathfrak{t}u_{D_1}(k(\omega_j))\|_{L^2 (\partial \Omega) }<\Lambda^\ast .
\]
\end{corollary}

The inverse problem we consider in the present paper was already studied, in the case of smooth star-shaped subdomains with respect to some fixed point, by Ammari and Triki \cite{AT}. For a similar problem with a single boundary measurement we refer to \cite{Ch3} where a Lipschitz stability estimate was established for a non monotone one-parameter family of unknown subdomains. The literature on the determination of an unknown part of the boundary is rich, but we just quote here \cite{ABRV,BCY,EHY} (see also the references therein). We note that
another multifrequency medium problem considers single observation for varying multiple wavenumbers \cite{BL2005,BT2010}. To have an overview, we recommend a recent review paper \cite{BLLT2015} which nicely summarizes the theoretical and numerical results in multifrequency inverse medium and source problems for acoustic Helmholtz equations and Maxwell equations.

The key step in our proof consists in reducing the original inverse problem to the one of determining an unknown part of the inner boundary from a single boundary measurement. For this last problem we provided in Section 2 a logarithmic stability estimate. This intermediate result is then used in Section 3 to prove Theorem \ref{theorem1.1}. Section 4 contains uniqueness results obtained from Theorem \ref{theorem1.1}.

The idea of reducing the original problem to the one of recovering the shape of an unknown inner part of the boundary was borrowed from the paper by Ammari and Triki \cite{AT}.

Our analysis combines both ideas from \cite{ABRV,AT} together with some recent results related to quantifying the uniqueness of continuation in various situations \cite{Ch1,Ch2,CT}.

This paper is completed by a last section in which we give an expansion of the solution of the BVP \eqref{1.1} in  the eigenfunction basis of the Neumann-Poincar\'e operator (shortened to NP operator in the rest of this paper) related to the Neumann-Green function.

\section{An intermediate estimate}

Pick $D\in \mathscr{D}_0(\varrho,\aleph)$ and let $\tilde{u}_D^0\in H^1(\Omega )$ satisfying $\tilde{u}_D^0=0$ in $\overline{D}$ and it is the variational solution of the BVP
\[
\left\{
\begin{array}{ll}
\Delta \tilde{u}=0\quad &\mbox{in}\; \Omega \setminus\overline{D} ,
\\
\tilde{u}=0 &\mbox{on}\; \partial D,
\\
\partial_\nu \tilde{u}=f &\mbox{on}\; \partial \Omega .
\end{array}
\right.
\]
As  $\tilde{u}_D^0\in C^\infty(\Omega \setminus\overline{D})$ by the usual interior regularity of harmonic functions, we can apply, for an arbitrary $\omega$,  $D\Subset \omega \Subset \Omega$, both the H\"older regularity theorem for Dirichlet  BVP in $\omega \setminus \overline{D}$ (\cite[Theorem 6.14 in page 107]{GT}) and the H\"older regularity theorem for Neumann BVP in $\Omega \setminus \overline{\omega}$ (\cite[Theorem 6.31 in page 128]{GT}). Therefore $\tilde{u}_D^0\in C^{2,\alpha}(\overline{\Omega}\setminus D)$.

Also, note that, according to the maximum principle and Hopf's maximum principle, $\tilde{u}_D^0\ge 0$.

Define then
\[
\tilde{u}_D=\tilde{u}_D^0-\frac{1}{|\partial \Omega|}\int_{\partial \Omega}\tilde{u}_D^0d\sigma(x).
\]

\begin{lemma}\label{lemmaB}
Let $D_1,D_2\in  \mathscr{D}_0(\varrho,\aleph)$ and set
\[
m_j=\frac{1}{|\partial \Omega|}\int_{\partial \Omega}\tilde{u}_D^0d\sigma(x).
\]
If $\partial D_1\cap \partial D_2\ne \emptyset$ then
\begin{equation}\label{B1}
|m_1-m_2| \le \|\tilde{u}_{D_1}-\tilde{u}_{D_2}\|_{L^\infty (\partial (D_1\cup D_2))}.
\end{equation}
\end{lemma}

\begin{proof}
As $\partial D_1\cap \partial D_2\subset \partial (D_1\cup D_2)$ and
\[
\tilde{u}_{D_1}-\tilde{u}_{D_2}= m_1-m_2 \quad \mbox{on}\; \partial D_1\cap \partial D_2,
\]
the expected inequality follows easily.
\qed
\end{proof}

Under the assumptions and notations of Lemma \ref{lemmaB}, we have from \eqref{B1}
\begin{equation}\label{B2}
\| \tilde{u}_{D_2}^0-\tilde{u}_{D_1}^0\|_{L^\infty (\partial (D_1\cup D_2))}\le 2\| \tilde{u}_{D_2}-\tilde{u}_{D_1}\|_{L^\infty (\partial (D_1\cup D_2))}.
\end{equation}

Let $\mathfrak{d}_1=(\varrho ,\aleph, \varrho_0,\aleph_0 ,\mathfrak{b})$. Then, checking carefully the result in \cite[Section 2.4]{Ch1} we find that there exist three constants $C=C(\mathfrak{d}_1)>0$, $c=c(\mathfrak{d}_1)>0$ and $\beta =\beta (\mathfrak{d}_1)$ so that, for any $0<\epsilon <1$ and $D_1,D_2\in \mathscr{C}_1(\mathfrak{d}_1)$, we have
\begin{align}
C\| \tilde{u}_{D_2}-\tilde{u}_{D_1}&\|_{L^\infty (\partial (D_1\cup D_2))}\label{B3}
\\
&\le \epsilon ^\beta\| \tilde{u}_{D_2}-\tilde{u}_{D_1}\|_{C^{1,\alpha} (\overline{\Omega}\setminus (D_1\cup D_2))}+e^{c/\epsilon}\| \tilde{u}_{D_2}-\tilde{u}_{D_1}\|_{H^1 (\partial \Omega) }.\nonumber
\end{align}
Let us note that \cite[Proposition 2.30]{Ch1} holds for an arbitrary bounded Lipschitz domain (we refer to \cite{BC} or \cite{Ch2} for a detailed proof of this improvement).

\begin{proposition}\label{propositionB}
Set $\mathfrak{d}=(\varrho, \aleph,\delta)$. There exists $C=C(\mathfrak{d})$ so that, for any $D\in \mathscr{D}_1(\mathfrak{d})$, we have
\[
\|\tilde{u}_D\|_{C^{2,\alpha}(\overline{\Omega}\setminus D)}\le C.
\]
\end{proposition}

\begin{proof}
Let $0\le g\in C^{2,\alpha}(\partial \Omega)$ and denote by $w\in C^{2,\alpha}(\overline{\Omega}\setminus D_0)$ the solution of
the BVP
\[
\left\{
\begin{array}{ll}
\Delta w=0\quad &\mbox{in}\; \Omega  ,
\\
w=g &\mbox{on}\; \partial D_0,
\\
\partial_\nu w=f &\mbox{on}\; \partial \Omega .
\end{array}
\right.
\]
The existence of such function is guaranteed by the usual elliptic regularity for both Dirichlet and Neumann BVP's for the Laplace operator. Similar argument will be discussed hereafter.

In light of the fact that $\widetilde{u}_D^0-w$ is harmonic in $\Omega \setminus\overline{D}$, $\widetilde{u}_D^0-w=-w\le 0$ on $\partial D$ and $\partial_\nu (\widetilde{u}_D^0-w)=0$ on $\partial \Omega$ and, we find by applying the twice the maximum principle and Hopf's lemma that $(0\le )\widetilde{u}_D^0\le w$ in $\overline{\Omega}\setminus D$. Whence
\[
\|\tilde{u}_D^0\|_{C(\overline{\Omega}\setminus D)}\le \|w\|_{C(\overline{\Omega}\setminus D)}.
\]
Let $\mu =\inf(\delta ,\mbox{dist}(\Omega ,\Omega_0))/2$ and set
\[
U=\{x\in \Omega \setminus\overline{D};\; \mbox{dist}(x,\Omega_0)\le \mu\}.
\]
By \cite[Lemma 3.11 in page 118]{Ch0}, we have
\begin{equation}\label{B4}
\|\tilde{u}_D^0\|_{C^{2,\alpha}(U)}\le \mu^{-1}C(n)\|\tilde{u}_D^0\|_{C(\overline{\Omega}\setminus D)}\le \mu^{-1}C(n)\|w\|_{C(\overline{\Omega}\setminus D)}.
\end{equation}
Let $u_1\in C^{2,\alpha}(\overline{\Omega_0}\setminus D)$ be the solution of the BVP
\[
\left\{
\begin{array}{ll}
\Delta u_1=0\quad &\mbox{in}\; \Omega_0\setminus\overline{D}  ,
\\
u_1=\tilde{u}_D^0&\mbox{on}\; \partial \Omega_0\cup \partial D ,
\end{array}
\right.
\]
and $u_2\in C^{2,\alpha}(\overline{\Omega}\setminus \Omega_0)$ be the solution of the BVP
\[
\left\{
\begin{array}{ll}
\Delta u_2=0\quad &\mbox{in}\; \Omega\setminus\overline{\Omega_0}  ,
\\
\partial_\nu u_2=\partial_\nu \tilde{u}_D^0 &\mbox{on}\; \partial \Omega\cup \partial \Omega_0 .
\end{array}
\right.
\]
A careful examination of the classical Schauder estimates in \cite[Chapter 6]{GT}, for both Dirichlet and Neumann problems, we see that  the different constants only depend on $C^{2,\alpha}$ parameters of the domain. Therefore
\begin{align}
&\|u_1\|_{C^{2,\alpha}(\overline{\Omega_0}\setminus D)}\le C(\mathfrak{d})\|u_1\|_{C^{2,\alpha}(\partial \Omega_0\cup \partial D)}, \label{B5}
\\
&\|u_2\|_{C^{2,\alpha}(\overline{\Omega}\setminus \Omega_0)}\le C(\mathfrak{d})\|\partial_\nu u_2\|_{C^{1,\alpha}(\partial \Omega\cup \partial \Omega_0)}. \label{B6}
\end{align}
We put \eqref{B4} in \eqref{B5} and \eqref{B6} in order to get
\[
\|\tilde{u}_D^0\|_{C^{2,\alpha}(\overline{\Omega}\setminus D)}\le C.
\]
The expected inequality follows then by noting that
\[
\|\tilde{u}_D\|_{C^{2,\alpha}(\overline{\Omega}\setminus D)}\le 2\|\tilde{u}_D^0\|_{C^{2,\alpha}(\overline{\Omega}\setminus D)}.
\]
The proof is then complete.
\qed
\end{proof}

Let $w\in H^2(\Omega \setminus\overline{\Omega_0})$. We have from the usual interpolation inequalities and trace theorems, with $\mathfrak{d}=(\varrho,\aleph )$,
\[
\|w\|_{H^1(\partial \Omega )}\le C(\mathfrak{d})\|w\|_{L^2(\partial \Omega )}^{1/3}\|w\|_{H^{3/2}(\partial \Omega )}^{2/3}\le C(\mathfrak{d})C_{\Omega \setminus\overline{\Omega_0}}\|w\|_{L^2(\partial \Omega )}^{1/3}\|w\|_{H^{2}(\Omega \setminus\overline{\Omega_0})}^{2/3} ,
\]
the constant $C_{\Omega \setminus\overline{\Omega_0}}$ only depends on $\Omega \setminus\overline{\Omega_0}$.

In light of this inequality, Proposition \ref{propositionB} and inequalities \eqref{B2} and \eqref{B3}, we can state the following result

\begin{theorem}\label{theoremB}
Set $\mathfrak{d}_1=(\varrho,\aleph,\varrho_0,\aleph_0,\delta ,\mathfrak{b})$. There exist three constants $C=C(\mathfrak{d}_1)>0$, $c=c(\mathfrak{d}_1)>0$ and $\beta =\beta (\mathfrak{d}_1)$ so that, for any $0<\epsilon <1$ and $(D_1,D_2)\in \mathscr{C}_1(\mathfrak{d}_1)$ with $\partial D_1\cap \partial D_2\ne \emptyset$, we have
\[
C\| \tilde{u}_{D_2}^0-\tilde{u}_{D_1}^0\|_{L^\infty (\partial (D_1\cup D_2))} \le \epsilon ^\beta+e^{c/\epsilon}\| \tilde{u}_{D_2}-\tilde{u}_{D_1}\|_{L^2 (\partial \Omega) }^{1/3}.
\]
\end{theorem}

\begin{theorem}\label{theoremA}
Let $D\in \mathscr{D}_0(\varrho,\aleph)$. For any sequence $(k_j)$ in $(0,\infty)$ such that $\lim_{j\rightarrow\infty}k_j=\infty$, we have
\[
\lim_{j\rightarrow \infty}\|\mathfrak{t}u_D(k_j)-\mathfrak{t}\tilde{u}_D\|_{L^2(\partial \Omega )}=0.
\]
In particular, $(\mathfrak{t}u_D(k_j))\in \ell ^\infty (L^2(\partial \Omega ))$.
\end{theorem}

\begin{proof}
Let
\[
W=\left\{w\in H^1(\Omega \setminus \overline{D});\; \int_{\partial \Omega}w(x)d\sigma(x)=0\right\}.
\]
$W$ is a closed subspace of $H^1(\Omega \setminus \overline{D})$ and the norm $\|\nabla w\|_{L^2(\Omega \setminus  \overline{D})}$ is equivalent on $W$ to the norm $\|w\|_{H^1(\Omega \setminus \overline{D})}$. Moreover, for any $w\in V$, $w_{|\Omega \setminus \overline{D}}\in W$.

Clearly, $v_D(k)=u_D(k)-\tilde{u}_D$ is the variational solution of the BVP
\[
\left\{
\begin{array}{ll}
\mbox{div}\left(\left( k_0+(k-k_0)\chi_D\right)\nabla v\right)=-\mbox{div}\left(\left( k_0+(k-k_0)\chi_D\right)\nabla \tilde{u}_D\right)\quad &\mbox{in}\; \Omega ,
\\
\partial_\nu v=0 &\mbox{on}\; \partial \Omega .
\end{array}
\right.
\]
As
\[
\int_{\partial \Omega}u_D(k)d\sigma (x)=\int_{\partial \Omega}\tilde{u}_Dd\sigma (x)=0,
\]
we have $v_D(k)\in V$.

By Green's formula, for any $w\in V$, we get
\begin{align}
-k_0\int_{\Omega \setminus\overline{D}}\nabla v_D(k)\cdot \nabla wdx&-k\int_D\nabla v_D(k)\cdot \nabla wdx\label{0.1}
\\
&= k_0\int_{\Omega \setminus\overline{D}}\nabla \tilde{u}_D\cdot \nabla wdx-\int_{\partial \Omega }fwd\sigma (x).\nonumber
\end{align}
Take in this identity $w=v_D(k)$ and make use Cauchy-Schwarz's identity in order to obtain
\begin{align}
k\|\nabla &v_D(k)\|_{L^2(D)}^2+k_0\|\nabla v_D(k)\|_{L^2(\Omega \setminus\overline{D})}^2\label{0.2}
\\
&\le k_0\|\nabla v_D(k)\|_{L^2(\Omega \setminus\overline{D})}\|\nabla \tilde{u}\|_{L^2(\Omega \setminus\overline{D})}+\|f\|_{L^2(\partial \Omega )}\|\mathfrak{t}v_D(k)\|_{L^2(\partial \Omega )}.\nonumber
\end{align}
But the trace operator $w\in H^1(\Omega \setminus\overline{D})\mapsto w_{|\partial \Omega}\in L^2(\partial \Omega)$ is bounded. Hence there exists a constant $C_0>0$, depending on $\Omega$ and $D$ so that
\[
\|w\|_{L^2(\partial \Omega )}\le C_0\|\nabla w\|_{L^2(\Omega \setminus\overline{D})},\quad w\in W.
\]
Therefore
\begin{equation}\label{0.3}
\|\nabla v_D(k)\|_{L^2(\Omega \setminus\overline{D})}\le \|\nabla \tilde{u}\|_{L^2(\Omega \setminus\overline{D})}+C_0k_0^{-1}\|f\|_{L^2(\partial \Omega )}:=M.
\end{equation}
This in \eqref{0.2} entails
\begin{equation}\label{0.4}
\|\nabla v_D(k)\|_{L^2(D)}\le \frac{M}{\sqrt{k}}.
\end{equation}

Pick $(k_j)$ a sequence in $(0,\infty)\setminus\{k_0\}$ so that $k_j\rightarrow \infty$ as $j\rightarrow \infty$. Under the temporary notation $v_j=v_D(k_j)$, we have from \eqref{0.3} and \eqref{0.4} that $v_j$ is bounded in $H^1(\Omega )$ and $\nabla v_j \rightarrow 0$ in $L^2(D)$ when $j\rightarrow \infty$. Subtracting if necessary a subsequence, we may assume that $v_j\rightarrow \overline{v}\in H^1(\Omega )$, strongly in $H^{3/4}(\Omega )$ and weakly in $H^1(\Omega )$. In consequence, $\nabla \overline{v}=0$ in $D$.

Define
\[
W_0=\{ w\in W;\; w=0\; \mbox{on}\; \partial D\}.
\]
An extension by $0$ of a function in $W_0$ enables us to consider $W_0$ as a closed subspace of $V$.

It is not hard to see that \eqref{0.1} yields
\[
-k_0\int_{\Omega \setminus\overline{D}}\nabla v_j\cdot \nabla wdx= k_0\int_{\Omega \setminus\overline{D}}\nabla \tilde{u}_D\cdot \nabla wdx-\int_{\partial \Omega }fwd\sigma (x),\quad w\in W_0.
\]
Passing to the limit when $j\rightarrow \infty$, we obtain
\[
-k_0\int_{\Omega \setminus\overline{D}}\nabla \overline{v}\cdot \nabla wdx= k_0\int_{\Omega \setminus\overline{D}}\nabla \tilde{u}_D\cdot \nabla wdx-\int_{\partial \Omega }fwd\sigma (x),\quad w\in W_0.
\]
Taking in this identity an arbitrary $w\in C_0^\infty (\Omega \setminus\overline{D})$, we find that $\overline{v}$ is harmonic in $\Omega \setminus\overline{D}$. Applying then generalized Green's function to deduce that $\partial_\nu \overline{v}=0$. On the other hand we know that $v_j$ converges  strongly to $\overline{v}$ in $H^{1/4}(\partial \Omega )$ (thank to the continuity of the trace operator). Therefore $\overline{v}\in V$ and hence $\overline{v}$ is identically equal to zero, implying in particular that $v_j$ converges  strongly to $0$ in $L^2(\partial \Omega )$.
\qed
\end{proof}

We get by combining Theorems \ref{theoremB} and \ref{theoremA} the following result.
\begin{theorem}\label{theoremC}
If $\mathfrak{d}_1=(\varrho,\aleph,\varrho_0,\aleph_0, \delta,\mathfrak{b})$, then there exist three constants $C=C(\mathfrak{d}_1)>0$, $c=c(\mathfrak{d}_1)>0$ and $\beta =\beta (\mathfrak{d}_1)$ so that, for any $0<\epsilon <1$ and $(D_1,D_2)\in \mathscr{C}_1(\mathfrak{d}_1)$ with $\partial D_1\cap \partial D_2\ne \emptyset$,  and for any sequence $(k_j)$ in $(0,\infty)$ such that $\lim_{j\rightarrow\infty}k_j=\infty$, we have
\[
C\| \tilde{u}_{D_2}^0-\tilde{u}_{D_1}^0\|_{L^\infty (\partial (D_1\cup D_2))} \le \epsilon ^\beta+e^{c/\epsilon}\left[\sup_j\| u_{D_2}(k_j)-u_{D_1}(k_j)\|_{L^2 (\partial \Omega) }\right]^{1/3}.
\]
\end{theorem}

\section{Proof of the main result}

If $\mathcal{U}$ is a bounded domain of $\mathbb{R} ^n$, we set
\[
\mathcal{U} ^\delta =\{x\in \mathcal{U} ;\; \mbox{dist}(x,\partial \mathcal{U} )> \delta\},\quad \delta >0.
\]
Define then
\[
\varkappa (\mathcal{U})=\sup\{\delta >0;\; \mathcal{U} ^\delta\ne \emptyset \}.
\]

We endow $C^{1,\alpha}(\overline{\mathcal{U}})$ with norm
\[
|u|_{1,\alpha} := \|\nabla u\|_{L^2(\mathcal{U} )^n}+[\nabla u]_\alpha ,
\]
where
\[
[\nabla u]_\alpha =\sup\left\{ \frac{|\nabla u(x)-\nabla u(y)|}{|x-y|^\alpha};\; x,y\in \overline{\mathcal{U}} ,\; x\neq y\right\}.
\]

For $0<\alpha <1$, $\eta >0$ and $M>0$, define $\mathscr{S}(\mathcal{U})=\mathscr{S}(\mathcal{U}, \alpha ,\eta ,M )$ by
\[
\mathscr{S}(\mathcal{U})=\{ u\in C^{1,\alpha}(\overline{\mathcal{U}});\; |u|_{1,\alpha} \le M,\; \|\nabla u\|_{L^\infty (\Gamma )}\ge \eta
\; \mbox{and}\; \Delta u=0\}.
\]
Denote by $\mathscr{U}=\mathscr{U}(\varrho, \aleph ,\mathfrak{h},\mathfrak{b})$ the set of bounded domains $\mathcal{U}$ of $\mathbb{R}^n$ that are of class $C^{0,1}$, with parameters $\varrho>0$, $\aleph >0$, and satisfy $\varkappa(\mathcal{U})\ge \mathfrak{h}>0$ and $d_g^\mathcal{U}\le \mathfrak{b}$.

\begin{theorem}\label{theoremLB}
Let $\tilde{\mathfrak{d}}=(\varrho, \aleph ,\mathfrak{h},\mathfrak{b}, \alpha, \eta ,M)$.There exist two constants $c=c(\tilde{\mathfrak{d}})>0$  so that, for any $\mathcal{U}\in \mathscr{U}$, $0<\delta <\mathfrak{h}$,  $x_0\in \mathcal{U} ^\delta$ and $u\in \mathscr{S}(\mathcal{U})$, we have
\[
e^{-e^{c/\delta}}\le \| u\|_{L^2(B(x_0,\delta ))}.
\]
\end{theorem}

\begin{proof}
We mimic the proof of \cite[Theorem 2.1]{CT} in which we substitute the three-ball inequality of $u$ by a three-ball inequality for $\nabla u$. If one examines carefully the proof of \cite[Theorem 2.1]{CT}, he can see that the different constants do not depend on $\mathcal{U}\in \mathscr{U}$ but only on $(\varrho, \aleph ,\mathfrak{h},\mathfrak{b})$. We obtain
\[
e^{-e^{c/\delta}}\le \|\nabla  u\|_{L^2(B(x_0,\delta ))}.
\]
This and Caccioppoli's inequality yield the expected inequality.
\qed
\end{proof}

Bearing in mind that \eqref{As} holds, the following corollary is a consequence of Theorem \ref{theoremLB} and Proposition \ref{propositionB}.

\begin{corollary}\label{corollaryLB}
Set $\mathfrak{d}=(\varrho, \aleph,\delta)$. Then there exist two constants $c=c(\mathfrak{d})>0$ so that, for any $D\in \mathscr{D}_1(\mathfrak{d})$, $x_0\in \Omega_0\setminus\overline{D}$,   we have
\[
e^{-e^{c/d_0}}\le \| \tilde{u}_D^0\|_{L^2(B(x_0,d_0/4 ))},
\]
where $d_0=d(x_0,\overline{D})$.
\end{corollary}

\begin{proof}[of Theorem \ref{theorem1.1}]
Set $\tilde{u}_j=\tilde{u}_{D_j}^0$, $j=1,2$ . According to the maximum principle, we have
\begin{equation}\label{E6}
\max_{\overline{D}_2\setminus D_1}|\tilde{u}_1| =\max_{\overline{D}_2\setminus D_1}|\tilde{u}_1 -\tilde{u}_2|=\max_{\partial (D_2\setminus \overline{D}_1)}|\tilde{u}_1 -\tilde{u}_2|\le \max_{\partial (D_1\cup D_2)}|\tilde{u}_1-\tilde{u}_2|.
\end{equation}

Pick $x_0\in \partial D_2$ so that
\[
\max_{x\in \partial D_2}d(x, \overline{D}_1)=d(x_0, \overline{D}_1):=\overline{d}_1.
\]

Noting that $\tilde{u}_1\ge 0$, we  apply Harnack's inequality (see \cite[Proof of Theorem 2.5 in page 16]{GT}) in order to get
\[
\max_{B(x_0, \overline{d}/4)}\tilde{u}_1\le 3^n\min_{B(x_0, \overline{d}/4)}\tilde{u}_1\le 3^n\max_{\overline{D}_2\setminus D_1}|\tilde{u}_1|.
\]

Whence
\[
\int_{B(x_0,\overline{d}/4)}\tilde{u}_1^2(x)dx\le \left|\mathbb{S}^{n-1}\right|\left(3\overline{d}_1/4\right) ^n\left(\max_{\overline{D}_2\setminus D_1}|\tilde{u}_1|\right)^2.
\]

This and the estimate in Corollary \ref{corollaryLB} yield, by changing if necessary the constant $c$,
\begin{equation}\label{E7}
e^{-e^{c/\overline{d}_1}}\le  \max_{\overline{D}_2\setminus D_1}|\tilde{u}_1|.
\end{equation}

We have similarly
\begin{equation}\label{E8}
e^{-e^{c/\overline{d}_2}}\le  \max_{\overline{D}_1\setminus D_2}|\tilde{u}_2|\le \max_{\partial (D_1\cup D_2)}|\tilde{u}_1-\tilde{u}_2|,
\end{equation}
where
\[
\overline{d}_2:=\max_{x\in \partial D_1}d(x, \overline{D}_2).
\]

In light of \eqref{E6}, \eqref{E7} and \eqref{E8} entail
\begin{equation}\label{E9}
e^{-e^{c/\overline{d}}}\le  \max_{\partial (D_1\cup D_2)}|\tilde{u}_1-\tilde{u}_2|.
\end{equation}
Here $\overline{d}=\max (\overline{d}_1,\overline{d}_2)=\mathbf{d}_H(\overline{D}_1,\overline{D_2})$.

Set
\[
\Lambda := \sup_j\| u_{D_2}(k_j)-u_{D_1}(k_j)\|_{L^2 (\partial \Omega) }.
\]
Then estimates in Theorem \ref{theoremC} in \eqref{E9} give
\begin{equation}\label{E10}
e^{-e^{c/\overline{d}}}\le \epsilon ^\beta+e^{c/\epsilon}\Lambda^{1/3},\quad 0<\epsilon <1.
\end{equation}

Since the function $\epsilon \in (0,1)\rightarrow \epsilon ^\beta e^{-c\beta}$ is increasing, if $\Lambda <e^{-3c}:=\Lambda_0$ then we can take $\epsilon$ in \eqref{E10} so that $\epsilon ^\beta e^{-c\beta}=\Lambda^{1/3}$. A straightforward computation shows that $\epsilon \le 3(c+\beta)|\ln \Lambda|^{-1}$. Modifying if necessary $c$ in \eqref{E10}, we obtain
\begin{equation}\label{E11}
e^{-e^{c/\overline{d}}}\le |\ln \Lambda|^{-1}.
\end{equation}
Therefore, if $\Lambda < \min(\Lambda_0,e^{-e})$ then \eqref{E11} implies
\[
\overline{d} \le c\left(\ln \ln |\ln \Lambda |\right)^{-1}.
\]
The proof is then complete.
\qed
\end{proof}

\section{Uniqueness}

We first observe that the following uniqueness result is an immediate consequence of Theorem \ref{theorem1.1}.

\begin{corollary}\label{corollary4.1}
Assume that $\Omega$ is of class $C^{2,\alpha}$. Let $D_0\Subset D_j\Subset \Omega$ of class $C^{2,\alpha}$, $j=1,2$, so that $\partial D_1\cap\partial D_2\ne \emptyset$ and $\mathbf{d}_m(D_1,D_2)=\mathbf{d}_H(D_1,D_2)$. If $\mathfrak{t}u_{D_1}(k_\ell)=\mathfrak{t}u_{D_2}(k_\ell)$ for some sequence $(k_\ell)$ in $(0,\infty)\setminus\{k_0\}$ with $k_\ell \rightarrow \infty$ when $\ell \rightarrow \infty$, then $D_1=D_2$.
\end{corollary}

It is worth mentioning that this uniqueness result together with the analyticity of the mapping $k\mapsto \mathfrak{t}u_{D}(k)$ enable us to establish  an uniqueness result when $k$ varies in a subset of $(0,\infty)\setminus\{k_0\}$ possessing an accumulation point. Prior to that, we prove the following Lemma.

\begin{lemma}\label{lemma4.1}
Assume that $\Omega$ and $D$ are two Lipschitz domains of $\mathbb{R}^n$ so that $D\Subset \Omega$. Then the mapping
\[
k\in (0,\infty )\setminus\{k_0\}\mapsto \mathfrak{t}u_D(k)\in L^2(\partial \Omega )
\]
is real analytic.
\end{lemma}

\begin{proof}
Let $k\in (0,\infty)\setminus\{k_0\}$ and $|\ell |\le k/2$ so that $k+\ell (0,\infty )\setminus\{k_0\}$. Since $u_D(k+\ell)$ is the solution of the variational problem
\[
\int_\Omega \mathbf{a}_D(k+\ell)\nabla u_D(k+\ell)\cdot \nabla vdx=\int_{\partial \Omega}fvd\sigma (x),\quad v\in V,
\]
we have
\begin{equation}\label{4.1}
\min (k/2,k_0)\|u_D(k+\ell)\|_V\le \|\mathfrak{t}\|\|f\|_{L^2(\partial \Omega)}.
\end{equation}
Here $\|\mathfrak{t}\|$ denotes the norm of $\mathfrak{t}$ as bounded operator acting from $V$ into $L^2(\partial \Omega)$.

Next, let $w_D(k)\in V$ be the solution of the variational problem
\[
\int_\Omega \mathbf{a}_D(k)\nabla w_D(k)\cdot \nabla v=-\int_\Omega \chi_D\nabla u_D(k)\cdot \nabla v,\quad v\in V.
\]
Then elementary computations show that
\begin{align*}
\int_\Omega \mathbf{a}_D(k)\nabla&\left[u_D(k+\ell)-u_D(k)-\ell w_D(k)\right]\cdot \nabla v
\\
&=\ell\int_\Omega \chi_D\nabla \left[u_D(k)-u_D(k+\ell)\right]\cdot \nabla v,\quad v\in V.
\end{align*}
Hence
\begin{equation}\label{4.2}
\min (k,k_0)\|u_D(k+\ell)-u_D(k)-\ell w_D(k)\|_V\le |\ell |\|u_D(k)-u_D(k+\ell)\|_V.
\end{equation}
But
\[
\int_\Omega \mathbf{a}_D(k)\nabla \left[u_D(k+\ell)-u_D(k)\right]\cdot \nabla v=-\ell\int_\Omega \chi_D\nabla u_D(k+\ell)\cdot \nabla v,\quad v\in V.
\]
This entails
\[
\min (k,k_0)\|u_D(k+\ell)-u_D(k)\|_V\le |\ell |\|u_D(k+\ell)\|_V,
\]
which, combined with \eqref{4.1}, yields
\begin{equation}\label{4.3}
\|u_D(k+\ell)-u_D(k)\|_V\le C_0(k)|\ell |,
\end{equation}
with $C_0(k)=\left[\min (k/2,k_0)\right]^{-1}\|\mathfrak{t}\|\|f\|_{L^2(\partial \Omega)}$.

Putting \eqref{4.3} into \eqref{4.2}, we get
\[
\|u_D(k+\ell)-u_D(k)-\ell w_D(k)\|_V\le C_1(k)\ell ^2,
\]
where $C_1(k)=\left[\min (k,k_0)\right]^{-1}C_0(k)$.

In other words, we proved that the mapping
\[
 k\in (0,\infty )\setminus\{k_0\} \mapsto u_D(k)\in V
\]
is differentiable and its derivative $u'_D(k)$ is the solution of the variational problem
\[
\int_\Omega \mathbf{a}_D(k)\nabla u'_D(k)\cdot \nabla v=-\int_\Omega \chi_D\nabla u_D(k)\cdot \nabla v,\quad v\in V.
\]
Therefore, we have the a priori estimate
\begin{equation}\label{4.4}
\|u'_D(k)\|_V \le C\phi (k)^2,
\end{equation}
where we set $C=\|\mathfrak{t}\|\|f\|_{L^2(\partial \Omega)}$ and $\phi (k)=\left[\min (k,k_0)\right]^{-1}$.

Now an induction argument in $j$ shows that $k\in (0,\infty )\setminus\{k_0\} \mapsto u_D(k)\in V$ is $j$-times differentiable and $u^{(j)}_D(k)$ is the solution of the variational problem
\[
\int_\Omega \mathbf{a}_D(k)\nabla u^{(j)}_D(k)\cdot \nabla v=-j\int_\Omega \chi_D\nabla u^{(j-1)}_D(k)\cdot \nabla v,\quad v\in V.
\]
In light of this identity and \eqref{4.4} we show, again by using an induction argument in $j$, that
\[
\|u^{(j)}_D(k)\|_V \le Cj!\phi (k)^{j+1}.
\]
Consequently, if $|k-\ell |<\phi(k)^{-1}=\min (k,k_0)$, the series
\[
\sum_{j\ge 0}\frac{1}{j!}\|u^{(j)}_D(k)\|_V(k-\ell)^j
\]
converges and hence, thank to the completeness of $V$, the series
\[
\sum_{j\ge 0}\frac{1}{j!}u^{(j)}_D(k) (k-\ell)^j
\]
converges in $V$. That is we proved that $k\in (0,\infty )\setminus\{k_0\}\mapsto u_D(k)\in V$ is real analytic and, since $\mathfrak{t}\in \mathscr{B}(V,L^2(\partial \Omega ))$, we conclude that $k\in (0,\infty )\setminus\{k_0\}\mapsto \mathfrak{t}u_D(k)\in L^2(\partial \Omega )$ is also real analytic.
\qed
\end{proof}

In light of Corollary \ref{corollary4.1} and the fact that a real analytic function $F:(0,\infty )\setminus\{k_0\}\rightarrow V$ cannot vanish in a subset of $(0,\infty )\setminus\{k_0\}$ possessing an accumulation point in $(0,\infty )\setminus\{k_0\}$ without being identically equal to zero, we get the following uniqueness result.

\begin{corollary}\label{corollary4.2}
Assume that $\Omega$ is of class $C^{2,\alpha}$. Let $D_0\Subset D_j\Subset \Omega$ of class $C^{2,\alpha}$, $j=1,2$ so that $\partial D_1\cap\partial D_2\ne \emptyset$ and $\mathbf{d}_m(D_1,D_2)=\mathbf{d}_H(D_1,D_2)$. If $\mathfrak{t}u_{D_1}=\mathfrak{t}u_{D_2}$ in some subset of $(0,\infty )\setminus\{k_0\}$ having an accumulation point in $(0,\infty )\setminus\{k_0\}$, then $D_1=D_2$.
\end{corollary}

The uniqueness results in Corollaries \ref{corollary4.1} and \ref{corollary4.2} are different from those existing in the literature in the case of single measurement (compare with \cite[Theorems 4.3.2 and 4.3.5]{Is}).

\section{Expansion in the eigenfunction basis of the NP operator}

We introduce some definitions and results that will be proved in Appendix A.

Let $N$ be the Neumann-Green function on $\Omega$. That is $N$ obeys to the following properties, where $y\in \Omega$ is arbitrary,
\[
\Delta N(\cdot ,y)=\delta_y,\quad \partial _\nu N(\cdot ,y)_{|\partial \Omega}=0.
\]
We normalize $N$ so that
\[
\int_{\partial \Omega} N(x,y)d\sigma (x)=0,\quad y\in \Omega.
\]

Denote by $E$ the usual fundamental solution of the Laplacian in the whole space. That is
\[
E(x)=\left\{
\begin{array}{ll} -\ln |x|/(2\pi) &\mbox{if}\; n=2. \\ |x|^{2-n}/((n-2)\omega_n)\quad &\mbox{if}\; n\ge 3. \end{array}
\right.
\]
Here $\omega_n=|\mathbb{S}^{n-1}|$.

We establish in Appendix A that the Neumann-Green function is symmetric and has the form
\[
N(x,y)=E(x-y)+R(x,y),\quad x,y\in \Omega ,\; x\ne y.
\]
With $R\in C^\infty (\Omega \times \Omega )$ and $R(\cdot ,y)\in H^{3/2}(\Omega )$, $y\in \Omega$. 

Consider the integral operators acting on $L^2(\partial D)$ as follows
\[
\mathscr{S}_Df(x)=\int_{\partial D}N(x,y)f(y)d\sigma (y),\quad f\in L^2(\partial D),\; x\in \Omega .
\]
We will see later that $\mathscr{S}_D$ is extended to a bounded operator from $H^{-1/2}(\partial D)$ into the space
\[
\mathscr{H}_D=\{u\in H^1(\Omega );\; \Delta u=0\; \mbox{in}\; \Omega \setminus \partial D\; \mbox{and}\; \partial_\nu u=0\; \mbox{on}\; \partial \Omega\}
\]
endowed with the norm $\|\nabla u\|_{L^2(\Omega )}$. Note that according to the usual trace theorems $\partial_\nu u$ is an element of $H^{-1/2}(\partial \Omega )$.

We define the NP operator $\mathscr{K}_D$  as the integral operator acting on $L^2(\partial D)$ with weakly singular kernel
\[
L(x,y)=\frac{(x-y)\cdot \nu (y)}{|x-y|^n}+\partial _{\nu (y)}R(x,y),\quad x,y\in \partial D,\; x\ne y.
\]

Finally, we establish in Appendix A that $\mathcal{S}_D=\mathscr{S}_D{_{|\partial D}}$ defines an isomorphism from $H^{-1/2}(\partial D)$ onto $H^{1/2}(\partial D)$.

\begin{theorem}\label{theorem5.1}
Define successively, as long as the maximum is positive, the energy quotients
\[
\lambda_j^+(D)=\max_{g\bot \{g_{D,0}^+,\ldots ,g_{D,j-1}^+\}}\frac{\|\nabla \mathscr{S}_Dg\|_{L^2(\Omega \setminus\overline{D})}^2-\|\nabla \mathscr{S}_Dg\|_{L^2(D)}^2}{\|\nabla \mathscr{S}_Dg\|_{L^2(\Omega )}^2}.
\]
Here the orthogonality is  with respect to the scalar product $(\nabla \mathscr{S}_D\cdot |\nabla \mathscr{S}_D\cdot )_{L^2(\Omega )}$. The maximum is attained at $g_{D,j}^+\in H^{1/2}(\partial D )$.

Define similarly
\[
\lambda_j^-(D)=\min_{g\bot \{g_{D,0}^-,\ldots ,g_{D,j-1}^-\}}\frac{\|\nabla \mathscr{S}_Dg\|_{L^2(\Omega \setminus\overline{D})}^2-\|\nabla \mathscr{S}_Dg\|_{L^2(D)}^2}{\|\nabla \mathscr{S}_Dg\|_{L^2(\Omega )}^2}.
\]
The minimum is attained at $g_{D,j}^-\in H^{1/2}(\partial D)$.

The potentials $\mathscr{S}_Dg_{D,j}^\pm$ together with all $\mathscr{S}_Dh$, $h\in  \ker (\mathscr{K}_D\mathcal{S}_D)\subset H^{-1/2}(\partial D)$, are mutually orthogonal and complete in $\mathscr{H}_D$.
\end{theorem}

This eigenvalue variational problem is correlated to the eigenvalue problem of the NP operator $\mathscr{K}_D$.  Precisely, we have
\begin{corollary}\label{corollary5.1}
The spectrum of $\mathscr{K}_D$ consists in the eigenvalues $\mu_j^\pm(D)=-\lambda_j^\pm(D)/2$, $j\ge 1$, multiplicities included, together with possibly the point zero. The extremal functions $g_{D,j}^\pm$ are exactly the eigenvalues of $\mathscr{K}_D$.
\end{corollary}

Set $\varphi_{D,j}^\pm = \mathscr{S}_Dg_{D,j}^\pm/\|\nabla \mathscr{S}_Dg_{D,j}^\pm\|_{L^2(\Omega )}$, $j\ge 1$ and
\[
\mathscr{H}_D^\pm =\mbox{span}\{ \varphi_{D,j}^\pm;\; j\ge 1\}.
\]
Let $\varphi_{D,j}^0$, $1\le j \le \aleph$ if $0<\aleph < \infty$ and $j\ge 1$ if $\aleph =\infty$, be an orthonormal basis of $\mathscr{H}_D^0=\{ \psi =\mathscr{S}_Dh;\; h\in \ker (\mathscr{K}_D\mathcal{S}_D)\}$. Here $\aleph \in [0,\infty]$ is the dimension of $\mathscr{H}_D^0$.

For simplicity convenience we only treat the case $\aleph =\infty$. The results in the case $\aleph <\infty$ are quite similar.

The preceding theorem says that $\mathscr{H}_D=\mathscr{H}_D^+\oplus \mathscr{H}_D^-\oplus \mathscr{H}_D^0$.

We are now ready to give the expansion of the solution of the BVP \eqref{1.1} in the basis $\{\varphi_{D,j}^\epsilon ,\; j\in I,\; \epsilon\in\{+,-,0\}\}$, where we set $I =\{j\ge 1\}$.

\begin{proposition}\label{proposition5.1}
 Let $u_D(k)$ be the solution of the BVP \eqref{1.1}.Then
\\
(i) $u_D$ admits the following expansion, where $\tilde{u}_D$ is as in the beginning of Section 2,
\[
u_D(k)=\tilde{u}_D+ \sum_{\underset{\epsilon \in\{+,-,0\}}{j\in I}}A_{D,j}^\epsilon (k) \varphi_{D,j}^\epsilon,
\]
the coefficients $A_{D,j}^\epsilon (k)$ satisfies, for $j\in I$ and $\epsilon\in \{+,-,0\}$,
\[
 (k-k_0)\sum_{\underset{\eta \in\{+,-,0\}}{i\in I}}A_{D,i}^\eta (k) (\nabla \varphi_{D,i}^\eta |\nabla \varphi_{D,j}^\epsilon)_{L^2(D)}+k_0A_{D,j}^\epsilon (k)=k_0B_{D,j}^\epsilon ,
 \]
where
\[
B_{D,j}^\epsilon = (f|\varphi_{D,j}^\epsilon)_{L^2(\partial \Omega)}-(\nabla \tilde{u}_D|\nabla \varphi_{D,j}^\epsilon)_{L^2(\Omega \setminus\overline{D})}.
\]
(ii) $k\in (0,\infty )\setminus\{k_0\}\rightarrow A_{D,j}^\epsilon (k)$, $j\in I$ and $\epsilon\in \{+,-,0\}$, is real analytic. Moreover, for any $\tilde{k}\in (0,\infty )\setminus\{k_0\}$, there exists $\delta >0$ so that, for any $j\in I$, $\epsilon\in \{+,-,0\}$ and $|k-\tilde{k}|<\delta$,  the series
\[
\sum_{\ell \in \mathbb{N}}\frac{1}{\ell !}\frac{d^\ell}{dk^\ell}A_{D,j}^\epsilon (\tilde{k})(k-\tilde{k})^\ell
\]
converges.
\end{proposition}

\begin{proof}
(i) Recall that $u_D$ has the following decomposition $u_D(k)=\tilde{u}_D+v_D(k)$. Observing that $v_D(k)$ satisfies $\partial_\nu v_D(k)=0$ (as an element of $H^{-1/2}(\partial \Omega )$) and
\[
{\rm div}(\mathbf{a}_D(k)\nabla v_D(k))=-{\rm div}(\mathbf{a}_D(k)\tilde{u}_D)\quad \rm{in}\; \mathscr{D}'(\Omega ),
\]
we find by using the generalized Green's formula
\begin{align}
k(\nabla v_D(k)|\nabla \varphi)_{L^2(D)}&+k_0(\nabla v_D(k)|\nabla \varphi)_{L^2(\Omega \setminus\overline{D})}\label{5.16}
\\
&=k_0(f|\varphi)_{L^2(\partial \Omega)}-k_0(\nabla \tilde{u}_D|\nabla \varphi)_{L^2(\Omega \setminus\overline{D})},\quad \varphi \in H^1(\Omega ).\nonumber
\end{align}

We expand $v_D(k)$ in the basis $\{\varphi_{D,j}^\epsilon ,\; j\in I,\; \epsilon\in\{+,-,0\}\}$:
\[
v_D(k)=\sum_{\underset{\epsilon \in\{+,-,0\}}{j\in I}}A_{D,j}^\epsilon (k) \varphi_{D,j}^\epsilon.
\]
Taking $\varphi=\varphi_{D,j}^\epsilon$ in \eqref{5.16}, we obtain
\begin{equation}\label{5.17}
(k-k_0)(\nabla v_D(k)|\nabla \varphi_{D,j}^\epsilon)_{L^2(D)}+k_0A_{D,j}^\epsilon (k)=k_0B_{D,j}^\epsilon,
\end{equation}
for $j\in I$ and  $\epsilon\in \{+,-,0\}$.

 Whence
 \begin{equation}\label{5.21}
 (k-k_0)\sum_{\underset{\eta \in\{+,-,0\}}{i\in I}}A_{D,i}^\eta (k) (\nabla \varphi_{D,i}^\eta |\nabla \varphi_{D,j}^\epsilon)_{L^2(D)}+k_0A_{D,j}^\epsilon (k)=k_0B_{D,j}^\epsilon ,
 \end{equation}
 for $j\in I$ and $\epsilon\in \{+,-,0\}$.

 (ii) 
We know from the preceding section that $k\in (0,\infty )\setminus\{k_0\}\rightarrow v_D(k)\in V$ is real analytic. Then so is $k\in (0,\infty )\setminus\{k_0\}\rightarrow A_{D,j}^\epsilon (k)\in \mathbb{C}$ and $k\in (0,\infty )\setminus\{k_0\}\rightarrow (\nabla v_D(k)|\nabla \varphi_{D,j}^\epsilon)_{L^2(D)}\in \mathbb{C}$, $j\in I$, $\epsilon\in \{+,-,0\}$.

We get by taking successively the derivative in \eqref{5.17} with respect to $k$
\begin{align}
(k-k_0)(\nabla v_D^{(\ell)}(k)|\nabla \varphi_{D,j}^\epsilon)_{L^2(D)}&+k_0\frac{d^\ell}{dk^\ell}A_{D,j}^\epsilon (k)\label{5.19}
\\
&=-\ell (\nabla v_D^{(\ell-1)}(k)|\nabla \varphi_{D,j}^\epsilon)_{L^2(D)},\nonumber
\end{align}
for $j\in I$, $\epsilon\in \{+,-,0\}$ and $\ell \in \mathbb{N}\setminus\{0\}$.

The choice of $k=\tilde{k}$ in \eqref{5.19} entails
\begin{align}
k_0\frac{d^\ell}{dk^\ell}A_{D,j}^\epsilon (\tilde{k})=-(\tilde{k}-k_0)&(\nabla v_D^{(\ell)}(\tilde{k})|\nabla \varphi_{D,j}^\epsilon)_{L^2(D)}+\label{5.20}
\\
&-\ell (\nabla v_D^{(\ell-1)}(\tilde{k})|\nabla \varphi_{D,j}^\epsilon)_{L^2(D)},\nonumber
\end{align}
for $j\in I$, $\epsilon\in \{+,-,0\}$ and $\ell \in \mathbb{N}\setminus\{0\}$.

We have
\[
\left|(\nabla v_D^{(\ell)} (\tilde{k})|\nabla \varphi_{D,j}^\epsilon)_{L^2(D)}\right|\le \|\nabla v_D^{(\ell )}(\tilde{k})\|_{L^2(\Omega)},\quad \ell \in \mathbb{N}.
\]
On the other hand, the series
\[
\sum_{\ell \in \mathbb{N}}\frac{1}{\ell !}v_D^{(\ell )}(\tilde{k})(k-\tilde{k})^\ell
\]
converges in $V$ provided that $|k-\tilde{k}|\le \delta$, for some $\delta$. Therefore, in light of \eqref{5.20}, we can assert that the series
\[
\sum_{\ell \in \mathbb{N}}\frac{1}{\ell !}\frac{d^\ell}{dk^\ell}A_{D,j}^\epsilon (\tilde{k})(k-\tilde{k})^\ell
\]
also converges whenever $|k-\tilde{k}|\le \delta$.
The proof is then complete.
\qed
\end{proof}

\begin{remark}
{\rm
Unfortunately, computing all the terms $(\nabla \varphi_{D,i}^\eta |\nabla \varphi_{D,j}^\epsilon)_{L^2(D)}$ seems to be not possible, especially for $\eta \ne \epsilon$. Let us compute those equal to zero. As $\varphi_{D,j}^\pm$ is a solution of a minimisation problem, we obtain in a standard way
\[
\lambda_k^\pm (D)\int_\Omega \nabla \varphi_{D,j}^\pm \cdot \nabla \varphi dx= \int_{\Omega\setminus\overline{D}} \nabla \varphi_{D,j}^\pm \cdot \nabla \varphi dx-\int_D \nabla \varphi_{D,j}^\pm \cdot \nabla \varphi dx,
\]
for any $\varphi \in \mbox{span}\{\varphi_{D,1}^\pm,\ldots \varphi_{D,j-1}^\pm\}^\bot$.

In particular,
\begin{equation}\label{5.22}
0= \int_{\Omega\setminus\overline{D}} \nabla \varphi_{D,j}^\pm \cdot \nabla \varphi_{D,i}^\pm dx-\int_D \nabla \varphi_{D,j}^\pm \cdot \nabla \varphi_{D,i}^\pm dx,\quad i>j.
\end{equation}
But
\[
0=\int_{\Omega\setminus\overline{D}} \nabla \varphi_{D,j}^\pm \cdot \nabla \varphi_{D,i}^\pm dx+\int_D \nabla \varphi_{D,j}^\pm \cdot \nabla \varphi_{D,i}^\pm dx,\quad i>j.
\]
This and \eqref{5.22} yield
\begin{equation}\label{5.23}
(\nabla \varphi_{D,i}^\pm |\nabla \varphi_{D,j}^\pm)_{L^2(D)}=(\nabla \varphi_{D,i}^\pm |\nabla \varphi_{D,j}^\pm)_{L^2(\Omega\setminus\overline{D})}=0,\quad i>j.
\end{equation}

We have similarly
\[
0= \int_{\Omega\setminus\overline{D}} \nabla \varphi_{D,j}^0 \cdot \nabla \varphi dx-\int_D \nabla \varphi_{D,j}^0 \cdot \nabla \varphi dx,
\]
for any $\varphi \in \mathscr{H}_D^0$. Hence
\begin{equation}\label{5.24}
0= \int_{\Omega\setminus\overline{D}} \nabla \varphi_{D,j}^0 \cdot \nabla \varphi_{D,i}^0dx-\int_D \nabla \varphi_{D,j}^0 \cdot \nabla \varphi_{D,i}^0 dx,\quad i,j\in I.
\end{equation}
As before, we deduce from \eqref{5.24}
\begin{equation}\label{5.25}
(\nabla \varphi_{D,i}^0 |\nabla \varphi_{D,j}^0)_{L^2(D)}=(\nabla \varphi_{D,i}^0 |\nabla \varphi_{D,j}^0)_{L^2(\Omega\setminus\overline{D})}=0,\quad i,j\in I,\; i\ne j.
\end{equation}
}
\end{remark}

\section*{Appendix A: Spectral analysis of the NP operator}\label{se_appendix}

Prior to proceed to the spectral analysis, we define some integral operator with weakly singular kernels. Let $D\Subset \Omega$ be a bounded domain of $\mathbb{R}^n$, $n\ge 2$, of class $C^{1,\alpha}$, for some $0<\alpha <1$.

Denote by $\nu$ the unit normal outward vector field on $\partial \Omega$. Then a slight modification of the proof of \cite[Lemma 3.15, page 124]{Fo} yields
\[
|(x-y)\cdot \nu (x)|\le C|x-y|^{1+\alpha},\quad x,y\in \partial D,\; x\ne y.
\]
Here the constant $C$ only depends on $D$. Hence
\begin{equation}\label{e1}
\left| \frac{(x-y)\cdot \nu (x)}{|x-y|^n}\right|\le \frac{C}{|x-y|^{n-1-\alpha}},\quad x,y\in \partial D.
\end{equation}
Define the integral operator $K_D^\ast:L^2(\partial D)\rightarrow L^2(\partial D)$  by
\[
K_D^\ast f(x)=\int_{\partial D} \frac{(x-y)\cdot \nu (x)}{|x-y|^n}f(y)d\sigma(y),\quad f\in L^2(D).
\]
Estimate \eqref{e1} says that the kernel of $K_D^\ast $ is weakly singular and therefore it is compact (see for instance \cite[Section 2.5.5, page 128]{Tr}).

Note that $K_D^\ast$ is nothing but the adjoint of the operator $K:L^2(\partial D)\rightarrow L^2(\partial D)$ given as follows
\[
K_D f(x)=\int_{\partial D} \frac{(x-y)\cdot \nu (y)}{|x-y|^n}f(y)d\sigma(y),\quad f\in L^2(D).
\]
As $K_D^\ast$, $K_D$ is also an integral operator with weakly singular kernel and then it is also compact.

Denote by $E$ the usual fundamental solution of the Laplacian in the whole space. That is
\[
E(x)=\left\{
\begin{array}{ll} -\ln |x|/(2\pi) &\mbox{if}\; n=2. \\ |x|^{2-n}/((n-2)\omega_n)\quad &\mbox{if}\; n\ge 3. \end{array}
\right.
\]
Here $\omega_n=|\mathbb{S}^{n-1}|$.

Recall that the single layer potential $S_D$ is the integral operator with kernel $E(x-y)$:
\[
S_Df(x)=\int_{\partial D}E(x-y)f(y)d\sigma(y),\quad f\in L^2(\partial D),\; x\in \mathbb{R}^n \setminus \partial D.
\]

Before stating a jump relation satisfied by $S_Df$, we introduce the notations, where $w\in C^1(\mathbb{R}^n)$ and $x\in \partial D$,
\begin{align*}
&w_{|\pm}(x)=\lim_{t\searrow 0}w(x\pm t\nu (x)),
\\
&\partial_\nu w(x)_{|\pm}=\lim_{t\searrow 0}\nabla w(x\pm t\nu (x))\cdot \nu (x).
\end{align*}

For any $f\in L^2(\partial D)$, $\partial_\nu S_Df(x)_{|\pm}$ exists as an element of $L^2(\partial D)$ and the following jump relation holds
\begin{equation}\label{e2}
\partial_\nu S_Df(x)_{|\pm}(x)=\left(\pm \frac{1}{2}+K_D^\ast \right)f(x)\quad \mbox{a.e.}\; x\in \partial \Omega .
\end{equation}
We refer for instance to \cite[Theorem 2.4 in page 16]{AK} and its comments.

For $y\in \Omega$, consider the following BVP
\begin{equation}\label{e3}
\left\{
\begin{array}{ll}
\Delta R=0\quad &\mbox{in}\; \Omega ,
\\
\partial_\nu R=-\partial_\nu E(\cdot -y) &\mbox{on}\; \partial \Omega .
\end{array}
\right.
\end{equation}

As $\partial_\nu E(\cdot -y) \in L^2(\partial \Omega )$, in light of \cite[Theorem 2, page 204 and Remarks (b) page 206]{JK}, the BVP \eqref{e3} has a unique solution $R(\cdot ,y)\in H^{3/2}(\Omega )$ so that
\[
\int_{\partial \Omega} R(x,y)dx=\varkappa (y),
\]
where the constant $\varkappa (y)$ is to be determined hereafter.

Define then $N$ by
\[
N(x,y)=E(x-y)+R(x,y),\quad x,y\in \Omega ,\; x\ne y.
\]
The function $N$ obeys to the following properties, where $y\in \Omega$ is arbitrary,
\[
\Delta N(\cdot ,y)=\delta_y,\quad \partial _\nu N(\cdot ,y)_{|\partial \Omega}=0.
\]
We fix in the rest of this text $\varkappa (y)$ in such a way that
\begin{equation}\label{N}
\int_{\partial \Omega} N(x,y)d\sigma (x)=0.
\end{equation}
The function $N$ is usually called the Neumann-Green function.

Mimicking the proof of \cite[Lemma 2.14, page 30]{AK},  we get $N(x,y)=N(y,x)$, $x,y\in \Omega$, $x\ne y$. Hence $R(x,y)=R(y,x)$, $x,y\in \Omega$. By interior regularity for harmonic functions $R(\cdot ,y)$, $y\in \Omega$, belongs to $C^\infty (\Omega )$ and consequently $R(x,\cdot )$, $x\in \Omega$ is also in $C^\infty (\Omega )$ implying that $R\in C^\infty (\Omega \times \Omega )$.

Consider the integral operators acting on $L^2(\partial D)$ as follows
\[
\mathscr{S}_Df(x)=\int_{\partial D}N(x,y)f(y)d\sigma (y),\quad f\in L^2(\partial D),\; x\in \Omega .
\]
Clearly $\mathscr{S}_D=S_D+S_D^0$, where $S_D^0$ is the integral with (smooth) kernel $R$, i.e.
\[
S_D^0f(x)=\int_{\partial D} R(x,y)f(y)d\sigma (y,)\quad f\in L^2(\partial D),\; x\in \Omega .
\]
Using  \eqref{e2} we find that $\mathscr{S}_D$ obeys to the following the jump condition:
\begin{equation}\label{a4}
\partial_\nu \mathscr{S}_Df(x)_{|\pm}(x)=\left(\pm \frac{1}{2}+\mathscr{K}_D^\ast \right)f(x)\quad \mbox{a.e.}\; x\in \partial \Omega .
\end{equation}
Here $\mathscr{K}_D^\ast =K_D^\ast +K_{D,0}^\ast$, where $K_{D,0}^\ast$ is the integral operator with kernel $\partial_{\nu (x)}R$, which is the dual of the integral operator $K_{D,0}$ whose kernel is $\partial_{\nu (y)}R$ (thank to the symmetry of $R$).

We get in particular that $\mathscr{K}_D^\ast:L^2(\partial D)\rightarrow L^2(\partial D)$ is compact. More specifically, $\mathscr{K}_D^\ast:L^2(\partial D)\rightarrow H^1(\partial D)$ is bounded (see for instance \cite[Theorem 2.11, page 28]{AK}).

We defined in Section 5 $\mathcal{S}_D=\mathscr{S}_D{_{|\partial D}}$ that we consider as a bounded operator on $L^2(\partial D)$ and set
\[
\mathcal{H}_D=\{u\in H^1(\Omega );\; \Delta u=0\; \mbox{in}\; \mathscr{D}'(\Omega \setminus \partial D)\}.
\]
Define on $\mathcal{H}_D$ the positive hermitian form
\[
(u|v)_{\mathcal{H}_D}=\int_\Omega \nabla u\cdot \overline{\nabla v}dx.
\]
The corresponding semi-norm is denoted by
\[
\|u\|_{\mathcal{H}_D}=(u|u)_{\mathcal{H}_D}^{1/2}.
\]

Let $u_-=u_{|D}$ and $u_+=u_{|\Omega \setminus\overline{D}}$. As $u_-\in H^1(D)$ and $\Delta u=0$ in $\mathscr{D}'(D)$, we know from the usual trace theorems that $\partial_\nu u_-\in H^{-1/2}(D)$. Similarly, we have $\partial_\nu u_+\in H^{-1/2}(\partial D)\cap H^{-1/2}(\partial \Omega )$. Therefore, according to generalized Green's formula, for any $u,v\in \mathcal{H}_D$, we have
\begin{align}
&\int_D\nabla u_-\cdot \nabla v_-=\langle \partial_\nu u_-|\overline{v_-}\rangle_{1/2,\partial D},\label{5.6}
\\
&\int_{\Omega \setminus\overline{D}}\nabla u_+\cdot \nabla v_+=-\langle \partial_\nu u_+|\overline{v_+}\rangle_{1/2,\partial D} +\langle \partial_\nu u_+|\overline{v_+}\rangle_{1/2,\partial \Omega}.\label{5.7}
\end{align}
The symbol $\langle \cdot |\cdot \rangle_{1/2,\Gamma}$ denotes the duality pairing between $H^{1/2}(\Gamma )$ and its dual $H^{-1/2}(\Gamma )$, with $\Gamma =\partial D$ or $\Gamma=\partial \Omega$.

Taking the sum side by side in inequalities \eqref{5.6} and \eqref{5.7}, we find
\begin{equation}\label{5.8}
(u|v)_{\mathcal{H}_D}= \langle \partial_\nu u_- -\partial_\nu u_+|\overline{v}\rangle_{1/2,\partial D} +\langle \partial_\nu u_+|\overline{v_+}\rangle_{1/2|\partial \Omega},
\end{equation}
where we used that $v_-=v_+=v$ in $\partial D$.

We apply \eqref{5.8} to $u=v=\mathscr{S}_Df$, with $f\in L^2(\partial D)$. Taking into account that $\partial_\nu u_+=0$ on $\partial \Omega$, we obtain
\[
\|\mathscr{S}_Df\|_{\mathcal{H}_D}=\langle \partial_\nu \mathscr{S}_Df_{|-}- \partial_\nu \mathscr{S}_Df_{|+}|\overline{\mathcal{S}_Df}\rangle_{1/2,\partial D}.
\]
This and the jump condition \eqref{a4} entail
\begin{equation}\label{5.9}
\|\mathscr{S}_Df\|_{\mathcal{H}_D}^2=\langle f|\overline{\mathcal{S}_Df}\rangle_{1/2,\partial D}=(f|\mathcal{S}_Df)_{L^2(\partial D)}.
\end{equation}
Here $(\cdot |\cdot )_{L^2(\partial D)}$ is the usual scalar product on $L^2(\partial D)$.

In other words, we proved that $\mathcal{S}_D:L^2(\partial D)\rightarrow L^2(\partial D)$ is strictly positive operator. Define, as in \cite{KPS}, $H_D^{-}$ to be the completion of $L^2(\partial D)$ with respect to the norm $\|\sqrt{\mathcal{S}_D}f\|_{L^2(\partial D)}=(f|\mathcal{S}_Df)_{L^2(\partial D)}$. Let $H_D^{+}=\mbox{R}(\sqrt{\mathcal{S}_D})$, the range of $\sqrt{\mathcal{S}_D}$, that can be regarded as the domain of the unbounded operator $\sqrt{\mathcal{S}_D}^{-1}$. We observe that $H_D^{+}$ is complete for the norm induced by the form $(\sqrt{\mathcal{S}_D}^{-1}f,\sqrt{\mathcal{S}_D}^{-1}f)_{L^2(\partial D)}$. Therefore $\mathcal{S}_D$ can be extended by continuity as an isomorphism, still denoted by $\mathcal{S}_D$, from $H_D^{-}$ onto $H_D^{+}$ and the pairing
\[
(\sqrt{\mathcal{S}_D}f|g)_{L^2(\partial D)}=(f|\sqrt{\mathcal{S}_D}g)_{L^2(\partial D)}
\]
defines a duality pairing between $H_D^{+}$ and $H_D^{-}$, with respect to the pivot space $L^2(\partial D)$.

Recall that the following closed subspace of $\mathcal{H}_D$ was introduced in Section 5.
\[
\mathscr{H}_D=\{u\in H^1(\Omega );\; \Delta u=0\; \mbox{in}\; \mathscr{D}'(\Omega \setminus \partial D)\; \mbox{and}\; \partial_\nu u_{|\partial \Omega} =0\}.
\]
Note that we have seen before that $\partial_\nu u_{|\partial \Omega}$ is an element of $H^{-1/2}(\partial \Omega )$. In that case \eqref{5.8} takes the form
\begin{equation}\label{5.10}
(u|v)_{\mathcal{H}_D}= \langle \partial_\nu u_- -\partial_\nu u_+|\overline{v}\rangle_{1/2,\partial D},\quad u\in \mathcal{H}_D,\; v\in \mathscr{H}_D.
\end{equation}

\begin{lemma}\label{lemma5.1}
For any $g\in H^{1/2}(\partial D)$, there exists a unique $v=\mathscr{E}_Dg\in \mathscr{H}_D$ so that $v_{|\partial D}=g$ and
\begin{equation}\label{5.11}
\|\mathscr{E}_Dg\|_{H^1(\Omega )}\le C\|g\|_{H^{1/2}(\partial D)},
\end{equation}
where the constant $C$ only depends on $\Omega$ and $D$.
\end{lemma}
\begin{proof}
Take $v\in H^1(\Omega )$ so that $v_{|D}=v_-$ and $v_{\Omega \setminus\overline{D}}=v_+$, where $v_-$ and $v_+$ are the respective variational solutions of the BVP's
\[
\left\{
\begin{array}{ll}
\Delta v_-=0\quad &\mbox{in}\; D ,
\\
v_-=g &\mbox{on}\; \partial D ,
\end{array}
\right.
\]
and
\[
\left\{
\begin{array}{ll}
\Delta v_+=0\quad &\mbox{in}\; \Omega\setminus\overline{D} ,
\\
v_+= g &\mbox{on}\; \partial D,
\\
\partial_\nu v_+=0 &\mbox{on}\; \partial \Omega .
\end{array}
\right.
\]
Clearly, $v\in \mathscr{H}_D$ and \eqref{5.11} holds.
\qed
\end{proof}

\begin{lemma}\label{lemma5.2}
There holds $H_D^{+}=H^{1/2}(\partial D)$.
\end{lemma}
\begin{proof}
Let $g\in H^{1/2}(\partial D)$ and $f\in L^2(\partial D)$. Apply then \eqref{5.10}, with $u=\mathscr{S}_Df$ and $v=\mathscr{E}_Dg$ in order to obtain
\[
(f|g)_{L^2(\partial D)}= (\mathscr{S}_Df|\mathscr{E}_Dg)_{\mathcal{H}_D}.
\]
Whence
\[
\left| (f|g)_{L^2(\partial D)}\right|\le C\|\mathscr{E}_Dg\|_{\mathcal{H}_D}\|\mathscr{S}_Df\|_{\mathcal{H}_D}\le C\|g\|_{H^{1/2}(\partial D)}\|\sqrt{\mathcal{S}_D}f\|_{L^2(\partial D)}.
\]
In other words, the linear form $f\in L^2(\partial D)\mapsto (f|g)_{L^2(\partial D)}$ is bounded for the norm $\|\sqrt{\mathcal{S}_D}f\|_{L^2(\partial D)}$. Therefore, according to Riesz's representation theorem, there exists $k\in L^2(\partial D)$ so that
\[
(f|g)_{L^2(\partial D)}=(f|\sqrt{\mathcal{S}_D}k).
\]
Since $f$ is arbitrary, we deduce that $g=\sqrt{\mathcal{S}_D}k$ and hence $g\in H_D^{+}$. That is $H^{1/2}(\partial D)\subset H_D^{+}$.

Conversely, as $f\in  H_D^{+}$ has the form $f=\mathcal{S}_Dg$ for some $g\in  H_D^{-}$, the coupling $(S_Dg|g):=(\sqrt{\mathcal{S}_D}g|\sqrt{\mathcal{S}_D}g)_{L^2(\partial D)}$ is well defined. Let $(g_k)$ be a sequence in $L^2(\partial D)$ converging to $g$ in the topology of $H_D^{-}$. We have
\[
\| \mathscr{S}_Dg_k-\mathscr{S}_Dg_\ell\|_{\mathcal{H}_D}=(\mathcal{S}_Dg_k-\mathcal{S}_Dg_\ell |g_k-g_\ell )_{L^2(\partial D)}.
\]
Hence  $(\mathscr{S}_Dg_k)$ is a Cauchy sequence in $\mathscr{H}_D$ which is complete with respect to the norm $\|\cdot \|_{\mathcal{H}_D}$. The limit $u$ of the sequence $(\mathscr{S}_Dg_k)$ satisfies $u_{|\partial D}=\mathcal{S}_Dg=f$ by the continuity of the trace map. Whence $f\in H^{1/2}(\partial D)$.
\qed
\end{proof}

As byproduct of the preceding proof we see that $\mathscr{S}_D$ is extended as a bounded operator from $H^{-1/2}(\partial D)$ onto $\mathscr{H}_D$ by setting
\[
\mathscr{S}_Df=\lim_k\mathscr{S}_Df_k,\quad f\in H^{-1/2}(\partial D ),
\]
where $(f_k)$ is an arbitrary sequence in $L^2(\partial D )$ converging to $f$ in $H^{-1/2}(\partial D)$. Moreover, we have
\[
u=\mathscr{S}_D\mathcal{S}_D^{-1}(u_{|\partial D}),\quad u\in \mathscr{H}_D.
\]

Introduce the double layer type operator
\[
\mathscr{D}_Df(x)=\int_{\partial D}\partial_{\nu (y)}N(x,y)f(y)d\sigma (y),\quad f\in L^2(\partial D),\; x\in \Omega \setminus \partial D.
\]

From \cite[Theorem 2.4, page 16]{AK}, we easily obtain, where $f\in L^2(D)$,
\begin{equation}\label{5.12}
\mathscr{D}_Df_{|\pm }=\left(\mp 1/2+\mathscr{K}_D\right)f\quad \mbox{a.e. on}\; \partial D.
\end{equation}

As in \cite[Lemma 2, page 154]{KPS}, our spectral analysis is based on the Plemelj's symmetrization principle. We have

\begin{lemma}\label{lemma5.3}
For $\mathcal{S}_D,\mathscr{K}_D:L^2(\partial D)\rightarrow L^2(\partial D)$, the following identity holds
\[
\mathscr{K}_D\mathcal{S}_D=\mathcal{S}_D\mathscr{K}_D^\ast.
\]
\end{lemma}

\begin{proof}
For $f\in \mathscr{D}(\partial D)$ and $x\in \Omega\setminus \partial D$, we have
\begin{align*}
\mathscr{D}_D\mathcal{S}_Df(x)&=\int_{\partial D}\partial_{\nu (y)}N(x,y)\left( \int_{\partial D}N(y,z)f(z)d\sigma(z)\right)d\sigma (y)
\\
&=\int_{\partial D}\left(\int_{\partial D}\partial_{\nu (y)}N(x,y)N(y,z)d\sigma (y)\right)f(z)d\sigma(z).
\end{align*}
As $N(x,\cdot )=N(\cdot ,x)$ and $N(\cdot ,z)$ are harmonic in $D$, Green's formula yields
\begin{align*}
\mathscr{D}_D\mathcal{S}_Df(x)&=\int_{\partial D}\left(\int_{\partial D}N(x,y)\partial_{\nu (y)}N(y,z)d\sigma (y)\right)f(z)d\sigma(z)
\\
&=\int_{\partial D}N(x,y)\partial_{\nu (y)}\left(\int_{\partial D}N(y,z)f(z)d\sigma (z)\right)d\sigma(y)
\\
&=\mathscr{S}_D[\partial_{\nu (y)}\mathscr{S}_Df_{|-}].
\end{align*}
In light of the jump conditions in \eqref{a4} and \eqref{5.12} we get
\[
\mathscr{D}_D\mathcal{S}_Df_{|+}=\left(-1/2+\mathscr{K}_D\right)\mathcal{S}_Df=\mathscr{S}_D[\partial_{\nu (y)}\mathscr{S}_Df_{|-}]_{|+}=\mathcal{S}_D(-1/2+\mathscr{K}_D^\ast)f.
\]
This and the fact that $\mathscr{D}(\partial D)$ is dense in $L^2(\partial D)$ yield the expected inequality.
\qed
\end{proof}

We recall the definition of the Shatten class, sometimes called also Shatten-von Neumann class. To this end, we consider a complex Hilbert space $H$. If $T:H\rightarrow H$ is a compact operator  and if $T^\ast :H\rightarrow H$ denotes its adjoint, then $|T|:=\sqrt{T^\ast T}:H\rightarrow H$ is positive and compact and therefore it is diagonalizable. The non-negative sequence $(s_k(T))_{k\ge 1}$ of eigenvalues of $|T|$ are usually called the singular values of $T$. For $1\le p<\infty$, if
\[
\sum_{k=1}^\infty \left[s_k(T)\right]^p<\infty
\]
we say that $T$ belongs to the Shatten class $S_p(H)$. It is worth mentioning that $S_p(H)$ is an ideal of $\mathscr{B}(H)$, $S_1(H)$ is known as the trace class and $S_2(H)$ corresponds to the Hilbert-Schmidt class.

\begin{lemma}\label{lemma5.4}
We have $\mathscr{K}_D\in S_p(L^2(\partial D))$, for any $p\ge (n-1)/\alpha$.
\end{lemma}

\begin{proof}
We first note that $\ell$, the kernel of $\mathscr{K}_D$, satisfies
\begin{equation}\label{5.15}
\ell (x,y) =O\left(|x-y|^{-(n-1)+\alpha}\right).
\end{equation}
Let $\omega$ be an open subset of $\mathbb{R}^{n-1}$ and $\rho$ continuous on $\omega \times \omega\setminus\{(x,x);\; x\in \omega\}$ and $\rho (x,y)=O\left(|x-y|^{-(n-1)+\alpha}\right)$, $x,y\in \omega$.

Let $1<q\le \alpha /(n-1-\alpha )$. Then elementary computations show that
\[
\sup_{x\in \omega}\int_\omega |\rho(x,y)|^qdx<\infty.
\]
From this we get, in light of \eqref{5.15} and using local cards and partition of unity, that $\ell \in L_x^\infty (\partial D;L_y^q(\partial D))$. Consequently, according to \cite[Theorem 1]{Ru}, $\mathscr{K}_D$ belogns to $S_{p}(L^2(\partial D))$, where $p=q/(q-1)$ is the conjugate exponent of $q$. We complete the proof by noting that the condition on $q$ yields $p\ge (n-1)/\alpha$.
\qed
\end{proof}

We remark that \cite[Theorem 4.2]{KPS} is obtained as an immediate consequence of  the abstract theorem \cite[Theorem 3.1]{KPS}. Since all the assumptions of \cite[Theorem 3.1]{KPS} hold in our case, we get, similarly to \cite[Theorem 4.2]{KPS},  Theorem \ref{theorem5.1}.

In Theorem \ref{theorem5.1} we anticipated the regularity of the extremal functions $g_{D,j}^\pm$. This result can be proved by following the same arguments as in \cite[page 164]{KPS}.

Following \cite{KPS}, $\lambda_j^\pm(D)$, $j\ge 1$, are called the eigenvalues of the spectral variational Poincar\'e problem: find those $\lambda$'s for which there exists $u\in \mathscr{H}_D$, $u\ne 0$, so that
\[
\int_{\Omega \setminus\overline{D}}\nabla u\cdot \overline{\nabla v}dx-\int_D\nabla u\cdot \overline{\nabla v}dx=\lambda \int_\Omega\nabla u\cdot \overline{\nabla v}dx,\; v\in \mathscr{H}_D.
\]
As we already mentioned in Section 5, this  variational eigenvalue problem is correlated to the eigenvalue problem for the NP operator $\mathscr{K}_D$. To see this, we observe that as straightforward consequence of \eqref{5.6}, \eqref{5.7} together with the jump condition \eqref{e2}, we have
\begin{align*}
\int_{\Omega \setminus\overline{D}}\nabla \mathscr{S}_Dg\cdot \overline{\nabla \mathscr{S}_Dh}dx&-\int_D\nabla \mathscr{S}_Dg\cdot \overline{\nabla \mathscr{S}_Dh}dx
\\
&= -2(\mathscr{K}_D^\ast \mathcal{S}_Dg|\mathcal{S}_Dh)_{L^2(\partial D)}=-2( \mathcal{S}_Dg|\mathscr{K}_D\mathcal{S}_Dh)_{L^2(\partial D)}.
\end{align*}
These identities are first established for $g,h\in L^2(\partial D)$ and then extended by density to $g,h\in H^{-1/2}(\partial D)$. Thus, we have in particular
\[
\frac{\|\nabla \mathscr{S}_Dg\|_{L^2(\Omega \setminus\overline{D})}^2-\|\nabla \mathscr{S}_Dg\|_{L^2(D)}^2}{\|\nabla \mathscr{S}_Dg\|_{L^2(\Omega )}^2}=\frac{-2( \mathcal{S}_Dg|\mathscr{K}_D\mathcal{S}_Dg)_{L^2(\partial D)}}{(\mathcal{S}_Dg|g)_{L^2(\partial D)}},
\]
with $g\in H^{-1/2}(\partial D)$.

These comments enable us to deduce Corollary \ref{corollary5.1} from Theorem \ref{theorem5.1}.

\section*{Acknowledgments}
CJ is supported by NSFC (key projects no.11331004, no.11421110002) and the Programme of Introducing Talents of Discipline to Universities (number B08018). MC is supported by the grant ANR-17-CE40-0029 of the French National Research Agency ANR (project MultiOnde). LS is supported by NSFC (No.91730304), Shanghai Municipal Education Commission (No.16SG01) and Special Funds for Major State Basic Research Projects of China (2015CB856003). This work started during the stay of MC at Fudan University on February 2017. He warmly thanks Fudan University for hospitality.

\end{document}